\RequirePackage{fix-cm}
\documentclass[smallextended, final]{svjour3}       

\smartqed  

\usepackage{geometry}
\usepackage[framemethod=default]{mdframed}
\usepackage{array} 
\usepackage{paralist}
\usepackage{verbatim}
\usepackage{subfig} 
\usepackage{hyphenat}
\usepackage{graphicx}
\usepackage{epsfig}
\usepackage{longtable}
\usepackage{color}
\usepackage{mathrsfs}
\usepackage{subfig}
\usepackage{amssymb}
\usepackage{xspace}
\usepackage{bbm}
\usepackage{amsmath}

\newmdenv[linecolor=black,backgroundcolor=white]{myframe}
\newlength\defaultparindent
\newcommand{\Real}{\ensuremath{\mathbb{R}}}
\def\AVI{{\rm AVI}}
\def\t{^\top}

\def\fnat{{\bf F}^{\rm nat}}
\newcommand{\norm}[1]{\left\lVert#1\right\rVert}
\def \eg{e.g.\@\xspace}
\newcommand{\ie}{i.e.\@\xspace}
\def\Ktilde{\widetilde K}
\def\Mtilde{\widetilde M}
\def\qtilde{\widetilde q}
\def\xtilde{\skew3\widetilde x}
\def\viproblem#1#2#3{\fbox
		 {\begin{tabular*}{0.95\textwidth}
{@{}l@{\extracolsep{\fill}}l@{\extracolsep{6pt}}l@{\extracolsep{\fill}}c@{}}
				#1 &#2 & $#3 $ 
			\end{tabular*}}}
\def\VI{{\rm VI}}
\def\norm#1{\|#1\|}
\def\Ebb{\mathbbm{E}}
\def\st{\mbox{subject to}}
\newcommand{\inv}{^{-1}}
\def\aur{\;\textrm{and}\;}
\def\ext{{\rm ext}}
\def\proj{\sqrt{\frac{n}{k}}R\t}

\begin{document}

\title{Dimensionality Reduction of Affine Variational Inequalities Using Random Projections
}

\author{Bharat Prabhakar         \and
        Ankur A. Kulkarni 
}

\institute{Bharat Prabhakar \at
              Department of Electrical Engineering, Indian Institute of Technology Bombay\\
                           \email{prabhakar.bharat@iitb.ac.in}          
           \and
           Ankur A. Kulkarni \at
             Systems and Control Engineering, Indian Institute of Technology Bombay\\
		\email{kulkarni.ankur@iitb.ac.in}
}

\date{Received: date / Accepted: date}

\maketitle

\begin{abstract}
We present a method for dimensionality reduction of an affine variational inequality (AVI) defined over a compact feasible region. Centered around the Johnson Lindenstrauss lemma~\cite{johnson1984extensions}, our method is a randomized algorithm  that produces with high probability an approximate solution for the given AVI by solving a lower-dimensional AVI. The algorithm allows the lower dimension to be chosen based on the quality of approximation desired. The algorithm can also be used as a subroutine in an exact algorithm for generating an initial point close to the solution. The lower-dimensional AVI is obtained by appropriately projecting the original AVI on a randomly chosen subspace. The lower-dimensional AVI is solved using standard solvers and from this solution an approximate solution to the original AVI is recovered through an inexpensive process. Our numerical experiments corroborate the theoretical results and validate that the algorithm provides a good approximation at  low dimensions and substantial savings in time for an exact solution.
\end{abstract}

\section{Introduction}
Technological advancements have enabled the collection and storage of a tremendously large amount of data. In parallel, the Internet has been changing the manner in which systems interact. As a result of this, an overwhelming amount of information is being generated and stored. A report by Harvard Magazine claims that, ``the total data accumulation of just the past two years -- a zettabyte -- dwarfs the prior record of human civilization"~\cite{bigdataharvard}, aptly justifying the name Big Data.

This paper concerns a challenge thrown up by high-dimensional problems in optimization and control that have arisen due to the growing prominence of such `big' or very large data sets~\cite{nocedal99numerical}. Exact algorithms for such problems can be computationally burdensome even if the algorithm has polynomial complexity. Whereas in contexts such as online optimization~\cite{wang2010fast} a  conceivable requirement could be not of the exact solution, but rather of a quick approximation, in the spirit of the Latin proverb \textit{bis das si cito das}\footnote{Twice you give, if you give quickly.}. 
We are motivated by this specific need where speed is of essence and accuracy can be sacrificed to some extent, if it means large savings in  time. 

This paper focuses on \textit{affine variational inequalities}. Variational inequalities are a versatile class of problems that generalize convex optimization~\cite{facchinei03finiteI}, saddle-point problems, Nash games and generalized Nash games~\cite{facchinei07ongeneralized,kulkarni12revisiting,kulkarni09refinement}, amongst others~\cite{facchinei03finiteI}. We consider affine variational inequalities with compact feasible regions. They are an important subclass, which include, \eg, constrained convex quadratic programming with compact feasible regions. Such quadratic programs are useful problems in their own right and are subproblems for the widely used \textit{sequential quadratic programming} algorithm for nonlinear programs~\cite{snopt,nocedal99numerical}. Affine variational inequalities also capture commonly used models for oligopolistic competition, such as Nash-Cournot games~\cite{facchinei03finiteI}. 

We present a dimensionality reduction technique for solving high-dimensional affine variational inequalities  \textit{approximately}. The method is 
\textit{probabilistic} in the sense that one can only guarantee that it works with ``high probability''. However, in exchange for this, we obtain a substantial saving in time. For  a polyhedral set $ K \subseteq \Real^n, $ a matrix $ M \in \Real^{n \times n} $ and a vector $ q \in \Real^n $,
an affine variational inequality (AVI) $ \AVI(K,M,q) $  is  the following problem,
$$ \viproblem{AVI$(K,M,q)$}{Find $ x\in K $ such that\ }{(y-x)\t(Mx+q)\geq 0 \qquad \forall \ y \in K.} $$
Given a high-dimensional (deterministic) AVI, the method derives a \textit{random} AVI from it which is low-dimensional. The lower-dimensional AVI is solved using standard solvers. Using the solution of the lower-dimensional AVI, a candidate solution of the high-dimensional AVI is generated through an inexpensive process. A probabilistic guarantee is obtained on the event that the error satisfies a bound.

The lower-dimensional AVI is obtained by \textit{projecting} the high\hyp{}dimensional problem on a subspace chosen uniformly at random. We implement this projection by multiplying by a suitably constructed random matrix. 
By the celebrated {\it Johnson Lindenstrauss Lemma} (JL lemma)~\cite{johnson1984extensions} we get that if a set of $ m $ points in a high dimensional space, are projected this way to a $ k $-dimensional subspace, then the probability that pairwise distances are at most $ \epsilon $-distorted, concentrates. This probability can be made to approach unity by appropriately choosing $ k $. E.g., for this probability to be $ 1-\delta $, we get $ k =O\left(\frac{\ln (m/\delta)}{\epsilon^2}\right) $. When applied to our setting, approximate distance preservation also allows for approximate preservation of inner products, which translates to an approximate solution of the AVI. 

Our main result is that with high probability the (deterministic) preimage (under the projection operation) of the solution  of the lower-dimensional AVI  approximately solves the given high-dimensional AVI. We recover a random approximation to this deterministic preimage by solving a linear program followed by a norm minimization quadratic program. Thanks to a remarkable result of Candes and Tao~\cite{candes2006near} we get that  the recovered solution approximates the required deterministic preimage with high probability. 

This framework also yields the following exact algorithm (with probabilistic guarantees): using the above technique one generates a point that is close to the true solution and this point is supplied to a standard solver as an initial point to obtain an exact solution. It is plausible that this method would improve the run-time for solvers that benefit from the ``local'' nature of the initial point. We have found this to be the case for the \textsc{PATH} solver~\cite{pathweb}. 

We emphasize that our algorithm does not assume any structure on the AVI such as monotonicity (the only assumption on the AVI is that the set $ K $ is compact). It seems plausible that further assumptions on the matrix $ M $ or on the set $ K $ may improve the theoretical results. 

Our numerical computations support the theoretical results. In particular, the exact algorithm (obtained by supplying the approximate solution as initial point) appears promising; in the examples we tried, considerable savings were obtained. For the theoretical results to hold we  require the lower-dimensional problem to be of size $ O(\ln \eta) $, where $ \eta $ is the number of extreme points of $ K. $ 
If $ K $ described by $ m $ inequality constraints, then for fixed $ m $, we have $ \ln \eta =O(\ln n)$ and in general $ \ln \eta $ is at most $ O(n) $, where $ n $ is dimension of the ambient space of $ K$.   
Although in the worst case $ O(\ln \eta) $ may not be significantly smaller than $ n $, in practice we have found that our algorithm performs well even for small values of the lower dimension. 

Conceptually speaking, this work exploits a delicate link between convex analysis and metric embeddings. The JL lemma may be viewed as a metric embedding result~\cite{goodman2004handbook}, with no obvious convex analytic properties. However in a Euclidean space, a metric embedding also implies the $\epsilon$-preservation of inner products, which under convexity allows optimality to be $\epsilon$-preserved.

This paper is organized as follows. Following the introduction, we present some background on the subject and define the desiderata of our algorithm. Section \ref{sec:random} formally introduces concepts pertaining to random projections required for our main results. The algorithm is introduced in Section \ref{sec:algorithm}. The proof of correctness is encompassed in  Section~\ref{sec:correctness}. We discuss some aspects pertaining to the algorithm, including the lower dimension, in Section~\ref{sec:remarks}. 
Section~\ref{sec:sim} contains numerical results and we conclude in Section~\ref{sec:disc}.

\subsection{Background}
For a closed convex set $ K \subseteq \Real^n$ and a continuous function $ F :\Real^n \rightarrow \Real^n, $ a variational inequality (VI) VI$(K, F)$ is the following problem~\cite{facchinei03finiteI},
$$ \viproblem{VI$(K,F)$}{\ Find $ x\in K $ such that\ }{(y-x)\t F(x)\geq 0 \qquad \forall \ y \in K.} $$

Solving a VI amounts to ensuring an \textit{angle condition}: the solution is a point $ x\in K $ such that $ F(x) $ makes 
an acute angle with all directions `$ y-x $' as $ y $ ranges over $ K. $ The solution of a convex optimization can also be written in this form and is a special case of the VI. The Nash equilibrium of a game is a simultaneous solution of several convex optimization problems and can also be captured by a VI~\cite{facchinei03finiteI}. Recent results~\cite{facchinei07ongeneralized,kulkarni12revisiting,kulkarni09refinement} have shown that certain equilibria of \textit{generalized} Nash games can also be captured by VIs. Besides these applications, VIs also generalize general equilibrium models, frictional contact problems and problems in finance, even as new applications continue to be considered. We refer the reader to~\cite{facchinei03finiteI} for more on this topic.

Solving $\VI(K,F)  $ is also equivalent to finding the zero of the function $ \fnat_K:\Real^n \rightarrow \Real^n $~\cite{facchinei03finiteI},
\begin{equation}
\fnat_K(x)=x-\Pi_K(x-F(x)), \label{eq:fnat}
\end{equation}
called the \textit{natural map} of the VI. $ \fnat_K $ is a continuous function and hence 
$ \norm{\fnat_K(x^*)} $ quantifies the ``quality'' of an approximate solution $ x^*$. Note that  $ \fnat_K $ is nonlinear except in rare (possibly uninteresting) cases. 

In practice, VIs are solved using a host of techniques, making use of the natural map as well as the \textit{normal map}~\cite{facchinei03finiteI}.  
We are not explicitly concerned with algorithms for solving VIs, since our method employs an off-the-shelf solver for solving the lower-dimensional AVI. We refer the reader to~\cite{facchinei03finiteII} for more on algorithms and to~\cite{dirkse93path,pathweb} for implementations.

Our work is essentially about efficiently solving VIs. 
There is a large body of work on solving large-dimensional optimization problems and variational inequalities with certain sparsity structure.  These include, for instance, Benders' decomposition~\cite{benders62partitioning} and the host of applications it spawned to stochastic optimization problems~\cite{vanslyke69lshaped,kulkarni10benders} and complementarity problems (see, \eg,~\cite{shanbhag06decomposition} and references therein). In a somewhat similar direction lie the series of works on splitting methods (see, \eg,~\cite{luo1992error} and related works), and more recent decomposition methods, \eg,~\cite{luna2014class}. 
These lines of research exploit the structure of the problem to decompose the larger problem into smaller subproblems. The algorithms so developed are exact algorithms with deterministic guarantees. In contrast, our algorithm does not assume any sparsity, but it is an approximate algorithm which works under probabilistic guarantees. Further, as mentioned in the introduction, it may be used as a subroutine in an exact algorithm.

To the best of our knowledge our work is the first application of random projections for dimensionality reduction in VIs. However, it has been preceded by many random projection-based algorithms in the computer science community (see, \eg, the monograph by Vempala~\cite{vempala04random}). In the operations research and control community, there are two applications of random projections  we are aware of. First~\cite{barman08note}, where a low-\textit{rank} approximation is used to approximately find the zero of a linear equation, and second~\cite{kulkarni09hilbert} where a similar approximation is used within a Newton method-based stochastic approximation. Our work differs from~\cite{barman08note} in two fundamental ways.  First, we seek a  low-dimensional approximation (rather than a low-rank one), and second, solving a VI reduces to solving the nonlinear equation~\eqref{eq:fnat} whereas~\cite{barman08note} critically relies on linearity.

\subsection{Problem definition}
We now formally define the problem we aim to solve. 
The objective of this paper is to solve  $ \AVI(K,M,q) $, i.e.,  where $ K \subseteq \Real^n$ is a compact polyhedron (a polytope), $ M\in \Real^{n\times n} $ and $ q \in \Real^n. $ The case we are interested in is where $ n $ is large. We seek an approximation algorithm that satisfies the following requirements
\begin{enumerate}
\item The most expensive step in the algorithm must involve solving a lower dimensional problem, \ie, the algorithm must operate in a lower-dimensional space. 
\item The lower dimensional problem should also be an affine variational inequality.
\item  The algorithm may be approximate, \ie, the candidate solution generated by the algorithm need not solve the problem exactly, but it should be a good approximation. 
\item The guarantee for the algorithm need not be deterministic, i.e., probabilistic guarantees on the solution would suffice.
\end{enumerate}

With this specified, we now proceed with the main contents of the paper, beginning with an overview of random projections in the following section.

\section{Random Projections} \label{sec:random}

In this section we review some results from the theory of random projections. Random projection is a particular case of an embedding of one metric space into another, and is as such is a part of a deeper mathematical study~\cite{goodman2004handbook}. We limit our survey here to operational aspects and to results relevant to our algorithm. 

\subsection{How to randomly project} \label{sec:construction}
Random projection involves the projection of vectors lying in a higher dimensional space to a {\it randomly} chosen lower-dimensional subspace. Note that this projection need not be Euclidean, \ie, the subspace need not be aligned with the basis vectors from the original space. A vector is projected  by multiplying the vector by a suitable random matrix; the choice of this matrix specifies the type of randomness introduced in the projection. 

Several different methods of constructing the random projection matrix have been studied in the literature. In this paper, we project the vectors on a {\it uniformly} random $k$-dimensional subspace. Such a subspace can be realized by choosing a uniformly random orthonormal matrix~\cite{vempala04random}. Below we show how such a matrix can be constructed.
 
\subsubsection{Constructing a uniformly random orthonormal matrix} \label{sec:constR}
An $n \times k$-dimensional real valued orthonormal matrix $R$ is \textit{uniformly random} if $R$ is uniformly distributed over the manifold, called Stiefel manifold~\cite{tulino2004random}, of real $n \times k$  matrices such that $R\t R = I$. 
We construct our  $n \times k$-dimensional random projection matrix $R$ as follows.
\begin{enumerate}
\item[{\bf R 1.}] Construct a matrix $ R_1 $ with each entry chosen independently from the distribution $N(0,\frac{1}{k})$.
\item[{\bf R 2.}] Orthonormalize the columns of $ R_1 $ using Gram-Schmidt process (QR-decomposition) and form the required matrix $R$ using these resultant vectors as columns. 
\end{enumerate}

We first observe that the matrix $ R_1 $ above is full rank almost surely. This is formalized in the following lemma.
\begin{lemma} \label{lem:rank}
The $n \times k$-dimensional random matrix $R_1$ (where $n>k$), obtained in the step {\bf R 1} in the construction described above, has a rank equal to $k$ with probability 1. 
\end{lemma}  
\begin{proof}
Consider the matrix $ Y $ formed by normalizing each column of $ R_1 $ to a unit vector (each column is not a $ 0 $-vector almost surely). It suffices to show that $ Y $ has rank $ k. $  
Let the columns of $Y$ be denoted by $Y^k = (Y_1,...Y_k)$. Recall that each vector $Y_i, 1\leq i \leq k$ is a uniformly distributed point on the unit sphere $S^{n-1}$~\cite{knuth1998art}. We prove the claim by induction. The first vector $Y_1$ is linearly independent with probability 1, as it is not a zero vector almost surely. Assume that $1<r \le k$ and that the first $r - 1$ vectors are linearly independent with probability 1. Then with probability 1 these $r -1$ vectors span a subspace V of dimension  $r -1$, which intersects $S^{n-1}$ in an $(r-2)$-dimensional ``subsphere" $S^{r-2}_V$. This subsphere forms a measure zero set under the uniform probability measure on $S^{n-1}$. Therefore the probability that $Y_r$ lies in this subsphere is zero. It follows that with probability 1 the vectors $Y_1, \dots Y_{r-1}, Y_r$ are linearly independent. Hence, by induction, rank($Y$) = $k$.
\end{proof}

From Lemma~\ref{lem:rank} it follows that the matrix $ R $ produced in Step {\bf R 2} is also of full column rank. To show that $ R $ is indeed uniformly distributed on the Steifel manifold, we invoke the ``real'' counterpart of Lemma \ref{lem:subspace} from~\cite{tulino2004random}.
\begin{lemma}[\cite{tulino2004random}] \label{lem:subspace}
Let $H$ be an $k \times n$ real valued standard Gaussian matrix with $n \ge k$. Denote its QR-decomposition by $H = \hat Q \hat R$. The upper triangular matrix $\hat R$ is independent of $\hat Q$, which is uniformly distributed over the manifold of $k \times n$ matrices such that $\hat Q \hat Q\t = I$. 

 \end{lemma} 
This completes the construction of a matrix $ R $ that is uniformly distributed on the Steifel manifold and has  full column rank.

\subsubsection{Projecting vectors and matrices} Consider a column vector $x  \in \mathbbm{R}^n$ (throughout this paper a vector is automatically to be assumed as a column vector), then the projection of $x$ is given by,
\begin{equation}
y = \sqrt{\frac{n}{k}}R\t x, \label{eq:y}
\end{equation}
The constant $\sqrt{\frac{n}{k}}$ is multiplied to ensure that the expected length of $y$ remains equal to that of $x$ (this property is needed by the Johnson Lindenstrauss lemma as we shall see ahead). Note that the value of constant being multiplied may vary depending upon the construction of the random projection matrix. 

For an arbitrary random vector $ y $ and an arbitrary random matrix $ R $, a vector $ x $ such that $ y = \sqrt{\frac{n}{k}}R\t x $ would, in general, be sample-path dependent (\ie, random). However if $ y $ is indeed a projection of a \textit{deterministic} vector $ x $, we call such 
an $ x $ its \textit{deterministic preimage}.
\begin{definition}
If a vector $y$ is a projection of some deterministic vector $ x $ as in \eqref{eq:y}, it's deterministic preimage is defined to be any  deterministic vector $x_o$, such that $\sqrt{\frac{n}{k}}R\t x_o = y$.
\end{definition}

The columns of a matrix $ Z \in \Real^{n \times d} $ can be thought of as a collection of $d$ $n$-dimensional vectors. Thus, the projection of $ Z $ is given by,
\begin{equation*}
Y = \sqrt{\frac{n}{k}}R\t Z,
\end{equation*}
where $Y \in \mathbbm{R}^{k \times d}$.
\subsection{The Johnson Lindenstrauss lemma}
The Johnson Lindenstrauss lemma~\cite{johnson1984extensions}  is a landmark result that shows that for any finite set of points there exists a mapping such that the distance between any pair of points is approximately equal to the distance between their images under the mapping. A key improvement \cite{frankl1988johnson} obtained later showed that if any finite set of points are projected \textit{randomly} by multiplication with a random orthonormal matrix, their pairwise distances are approximately preserved in the above sense. While the original JL lemma only provides the existence of a distance preserving mapping, at the expense of a probabilistic guarantee, the result from~\cite{frankl1988johnson} provides a \textit{construction} of this mapping.

For the purpose of our paper, we do not need this exact result but only its precursor below which concerns approximately preserving the norm of a single vector; this latter result can be found in~\cite{dasgupta2003elementary}. Since the lemma below differs from the exact result in~\cite{dasgupta2003elementary}, we  provide a sketch of the proof.

\begin{lemma} \label{lem:normjl}
Let $R \in \mathbbm{R}^{n \times k}$ be a random matrix as constructed in Section~\ref{sec:constR}, and let $f(u) = \sqrt{\frac{n}{k}}R\t u$ for $u \in \mathbbm{R}^n$. Then for any $0 < \epsilon < 1$,
\begin{equation*}
P\left((1-\epsilon)\|u\|^2 \le \|f(u)\|^2 \le (1+\epsilon)\|u\|^2\right)   \geq
1- 2e^{-(\epsilon^2/2 - \epsilon^3/3)k/2}
\end{equation*}
\end{lemma}
\begin{proof}
Let $ u \in \Real^n $ and without loss of generality assume $ \norm{u}=1. $ The vector $\proj u  $ is the projection of $ u $ on a $ k $-dimensional subspace chosen uniformly at random. As observed in~\cite{dasgupta2003elementary} the distribution of $ R\t u $
is the same as the distribution of the projection of a unit vector chosen uniformly at random to a fixed subspace. Thus, let $X_1,...X_n$ be $n$ independent Gaussian $N(0,1)$ random variables and define  $Y = \frac{1}{\|X\|}(X_1,...X_n)$. As argued in Lemma \ref{lem:rank}, $Y$ is a uniformly random vector lying on the unit sphere $S^{n-1}$. Fix the projection subspace to be one spanned by the first $k$ coordinates and let $Z \in \mathbbm{R}^k$ be the projection of $ Y $ on the first $ k $ coordinates. Define $L := \|Z\|^2$. Clearly, $ \mathbbm{E}[L] = k/n$. From~\cite{dasgupta2003elementary} we get,
\begin{align*}
P(L \le (1 - \epsilon)\mathbbm{E}[L]) \le e^{- \frac{k \epsilon^2}{4}} \qquad \aur  \qquad 
P(L \ge (1+\epsilon)\mathbbm{E}[L]) \le e^{- \frac{k( \epsilon^2/2 - \epsilon^3/3)}{2}}. 
\end{align*}
Now since $ R\t u $ has the same distribution as $Z $, $ \norm{f(u)}^2 $ has the same distribution as $ \frac{L}{\Ebb[L]}, $ from which the result follows.
\end{proof}

It is important to note that unlike in the deterministic result of Johnson and Lindenstrauss, the projection \textit{does not} depend on the vector being projected (the projection only has to be uniformly random). Further, since random projection here is a linear operator, the above result can easily be converted into a {\it distance preservation} result by applying it to a vector corresponding to the difference between two different vectors and then preserving its norm. By applying this projection to a set of vectors and then taking a union bound leads us to the complete version of JL lemma (see \eg,~\cite{vempala04random}). 

Finally note that, the probability in Lemma \ref{lem:normjl} is for the event that the norm of the projected vector lies within a factor of $(1 \pm \epsilon)$ of the norm of original vector. Whereas in {\it expectation}, the value of both these norms is the same, i.e., E$[\|f(u)\|^2] = \|u\|^2$. This underlines the nature of the JL lemma as a concentration phenomenon.
\subsection{Preserving inner products via JL lemma}

We now present the following lemma claiming that the inner products are approximately preserved under random projection. It follows from the norm preserving JL lemma.
\begin{lemma} \label{lem:innerjl}
Let $R \in \mathbbm{R}^{n \times k}$ be a random matrix as constructed in Section \ref{sec:constR}. Define a mapping $f(u) = \sqrt{\frac{n}{k}}R\t u$ for $u \in \mathbbm{R}^n$. Then for any two vectors $u, v \in \mathbbm{R}^n$, and $0<\epsilon <1$,
\begin{equation*}
P(|u \t v - f(u) \t f(v)| \le \epsilon \|u\| \|v\|) \geq 1 -  4e^{-(\frac{\epsilon^2}{2} - \frac{\epsilon^3}{3})k/2}
\end{equation*}
\end{lemma}
\begin{proof}
Consider two vectors $u' + v'$ and $u' - v'$, such that $u' = \frac{u}{\|u\|}$ and $v' = \frac{v}{\|v\|}$. We try to preserve the norms of both these vectors {\it simultaneously}. Let $A$  and $B$ denote the events 
\begin{align*}
A&:=  \{(1-\epsilon)\|u'+v'\|^2 \le \|f(u'+v')\|^2  \le (1+\epsilon)\|u'+v'\|^2\}  \\
B&:=  \{(1-\epsilon)\|u'-v'\|^2 \le \|f(u'-v')\|^2 \le (1+\epsilon)\|u'-v'\|^2\} 
\end{align*} 
We need to find a lower bound on $P(A \cap B)$, which is equivalent to finding an upper bound on $P(A^c \cup B^c)$. We have,
$P(A \cap B) =  1 - P(A^c \cup B^c) \geq  1 - P(A^c) - P(B^c),$
where the last inequality follows from the union bound on $ P(A^c \cup B^c) $. 
From Lemma \ref{lem:normjl}, we have  P($A$) $\geq 1- 2e^{-(\epsilon^2 - \epsilon^3)k/4}$ and P($B)\geq 1- 2e^{-(\epsilon^2 - \epsilon^3)k/4}$. 
Consequently, 
\begin{equation}
P(A \cap B) \geq 1 -  4e^{-(\frac{\epsilon^2}{2} - \frac{\epsilon^3}{3})k/2}. \label{eq:pacapb}
\end{equation}

Now since,
$ 4f(u') \t f(v') = \|f(u'+v')\|^2 - \|f(u'-v')\|^2$
we have from the definition of $ A,B $, that under the event $ A \cap B, $
\begin{align*}
4f(u') \t f(v') &\geq (1-\epsilon)\|u'+v'\|^2 - (1+\epsilon)\|u'-v'\|^2 \\ &= 4(u')\t v' - 2 \epsilon (\|u'\|^2 + \|v'\|^2).
\end{align*}
Since $u'$ and $v'$ are unit vectors, we get
$f(u') \t f(v') \ge (u' )\t v' - \epsilon.$
Similarly,  under the event $ A \cap B, $
\begin{align*}
4f(u') \t f(v')
 &\leq  4(u') \t v' + 2 \epsilon (\|u'\|^2 + \|v'\|^2),
\end{align*}
whereby,
$f(u') \t f(v') \le (u') \t v' + \epsilon.$
Since $ f $ is linear, $$\{|u \t v - f(u) \t f(v)| \le \epsilon \|u\| \|v\|\} \supseteq A \cap B.$$ Now using \eqref{eq:pacapb}, we get the result.
\end{proof}

The above lemma talks about preserving only one inner product. The following lemma generalizes the result and ensures the preservation of a finitely many inner products simultaneously.
\begin{lemma} \label{lem:innerpresmult}
For each  $ i = 1,\hdots, m $, let $ u_i,v_i $ be vectors $ \Real^n $. Also, let $R \in \mathbbm{R}^{n \times k}$ be a random matrix as constructed in Section \ref{sec:constR}, and define a mapping $f(x) = \sqrt{\frac{n}{k}}R\t x$ for $x \in \mathbbm{R}^n$. Then for $0< \epsilon <1$,
\begin{equation*}
P\left(|u_i\t v_i - f(u_i)\t f(v_i)| \le \epsilon \|u_i\| \|v_i\| \quad \forall  i=1,\hdots,m\right)  \geq1 -  4me^{-(\frac{\epsilon^2}{2} - \frac{\epsilon^3}{3})k/2} 
\end{equation*}
\end{lemma}
\begin{proof}
Let $A_i$ denote the event $ \{|u_i\t v_i - f(u_i)\t f(v_i)| \le \epsilon \|u_i\| \|v_i\|\} $ for $ i=1,\hdots,m. $
 Thus, from Lemma~2.4 we get for all $ i=1,\hdots,m, $
\begin{align*}
P(A_i^c) &\leq  4e^{-(\frac{\epsilon^2}{2} - \frac{\epsilon^3}{3})k/2}, 
\end{align*}
Similar to the analysis in the proof of Lemma~\ref{lem:innerjl}, by the union bound, 
\begin{align*}
P(A_1 \cap A_2 \cdots \cap A_m) &= 1 - P(A_{1}^c \cup A_{2}^c \cdots \cup A_{m}^c), \\
&\geq1 - [P(A_{1}^c) + P(A_{2}^c) + \cdots P(A_{m}^c)], \\
&\geq1 - 4me^{-(\epsilon^2/2 - \epsilon^3/3)k/2},
\end{align*}
as required.
\end{proof}

This completes the preliminaries pertaining to random projections. In the following section we present our algorithm.

\section{Algorithm} \label{sec:algorithm}
Algorithm A below is our dimensionality reduction algorithm for solving a typically high-dimensional affine variational inequality with a compact feasible region. It constructs a lower-dimensional AVI and using the solution obtained to this lower dimensional problem, generates an approximate solution to our original problem. Our algorithm probabilistically guarantees that the solution vector it generates, solves the original problem approximately.

\subsection*{Algorithm A}
\begin{enumerate}
\item \textbf{Input:} $ \AVI(K,M,q) $, \ie, $ M \in \Real^{n\times n}, q \in \Real^n, K=\{x | Ax \leq b\} $ where $ A \in \Real^{m \times n},b\in \Real^m. $
\item Fix the error parameter $\epsilon \in (0,1)$ and the success confidence parameter $\delta \in (0,1]$. Pick the lower dimension value $k$ such that $k \ge \frac{2\ln(\frac{4\eta}{\delta})}{(\frac{\epsilon^2}{2} - \frac{\epsilon^3}{3})}$, where $\eta = |\ext(K)|$ and $\ext(K)$ denotes the set of all extreme points of $K$.  
\item Construct an $n \times k$ dimensional random matrix $R$ as described in Section~\ref{sec:constR}. 
\item Construct the corresponding lower-dimensional \AVI($\widetilde{K}, \widetilde{M}, \widetilde{q}$). Define $\widetilde{q} = \sqrt{\frac{n}{k}}R\t q$ and $\widetilde{M} = R\t MR$, and $\widetilde{K}$ to be the projection of the polytope $K$, i.e., $\widetilde{K} = \{\widetilde{x} \in \mathbbm{R}^k | \widetilde{x} = \sqrt{\frac{n}{k}}R\t x, x \in K\}$. \label{step:proj}
\item Solve \AVI($\widetilde{K}, \widetilde{M}, \widetilde{q}$) to obtain a lower-dimensional solution $\widetilde{x}$. \label{step:detpre}
\item Solve the following $\ell_1$ norm minimization problem (can be converted into a linear program) to obtain a  vector $x^*$: \label{step:l1}
\[ x^* \in \arg \min_{x\in \Real^n} \|x\|_{1} \quad
\st \quad  \sqrt{\frac{n}{k}}R\t x = \widetilde{x}
 \]
\item Project $x^*$ on $K$ to obtain the final random approximate solution $x^{\#}$, i.e., \label{step:projk}
\begin{equation*}
x^{\#} = \Pi_K(x^*)\triangleq\underset{x \in K }{\text{arg min }} \|x^* - x\|^2
\end{equation*}
\item \textbf{Output:} $x^{\#}$.
\end{enumerate}
\noindent Our main result is as follows.

\begin{theorem} \label{thm:main} Let $\epsilon \in (0,1), \delta \in (0,1]$. Let $ \AVI(K,M,q) $ be an AVI where $ K $ is a compact polyhedron and consider Algorithm A. Suppose $\widetilde x$ is the lower-dimensional solution obtained in Step \ref{step:detpre} and let $x_o \in K$ be a deterministic preimage of $ \widetilde{x}$. Then the following claims hold.
\begin{enumerate}
\item  If the lower dimension $k$ satisfies  $k \ge \frac{2\ln(\frac{4\eta}{\delta})}{(\frac{\epsilon^2}{2} - \frac{\epsilon^3}{3})}$, then with  probability strictly greater than $(1-\delta)$,
$x_o$ solves $ \AVI(K,M,q) $ approximately, i.e., 
\begin{equation*}
(y - x_o)\t (q + Mx_o) \ge  \hat{\epsilon}, \quad \forall y \in K,
\end{equation*}
where $\hat{\epsilon} = - \frac{\left(\frac{\epsilon^2}{2} - \frac{\epsilon^3}{3}\right)n }{2\ln(\frac{4\eta}{\delta})} \|M\| B  - \epsilon \cdot  D \|q\| - \epsilon \cdot  D  \|M\| B$,
where $ \norm{M} $ denotes the $ \ell_2 $ induced norm of $ M $, $ D := \max_{x_1,x_2\in K} \norm{x_1-x_2} $ is the diameter of $ K $ and $ B :=\max_{x \in K} \|x\|$. 
\item Let $x^{\#}$ be the output generated by the algorithm A in Step 8, then with probability at least $1 - O(n^{-1/\alpha})$:
\begin{equation*}
\|x_o - x^{\#}\|_2 \le C B' \cdot (k/\ln(n))^{-1/2},
\end{equation*}
where $\alpha >0$ is a sufficiently small number (less than an absolute constant) and $C$ is a constant depending only on $\alpha$ and $B'=\max_{x \in K}\norm{x}_1 $.
\end{enumerate}
\end{theorem}

There are two parts to the above result. The first part shows that the deterministic preimage $ x_o $ of the lower-dimensional solution $ \widetilde{x} $ \textit{approximately} solves the given problem $ \AVI(K,M,q) $ in the sense that the inner product $ (y-x_o)\t(q+Mx_o) $ as $ y $ ranges over $ K $ is at least $ \hat{\epsilon}. $ Here $ \hat{\epsilon} $ is a negative number  that can be made arbitrarily close to zero  by choosing $ \epsilon $ to be small enough; for an exact solution to the AVI we would require $ \hat{\epsilon}\geq 0 $. Notice though that this result in itself is not computationally useful since $ x_o $
cannot be computed from a single sample of $ \widetilde{x} $. Each run of the algorithm produces a particular sample path: any process of ``inverting'' the relation $ \sqrt{\frac{n}{k}}R\t x_o = \widetilde{x} $ would in general yield an $ x_o $ that is sample-path  dependent (and hence $ R $-dependent). Steps 6-8 of the algorithm produce a random approximation $ x^\# $ to (the deterministic approximate solution) $ x_o $. The second part of the above theorem establishes that with high probability, these steps produce a good approximation.

In the following section, we prove Theorem \ref{thm:main} and establish the correctness of the algorithm.

\section{Correctness of the Algorithm} \label{sec:correctness}
There are two parts that need to be established to show the correctness of the algorithm. The first part showing  the point $ x_o $ is an approximate solution will be proved using the JL lemma.  The second part, showing that the recovered solution $ x^\# $ approximates $ x_o $ will be showed using a result of Candes and Tao~\cite{candes2006near}. Before we proceed with this analysis, we note that the lower-dimensional problem is indeed an AVI and that it admits a solution.

The lower dimensional problem is a variational inequality, $ \VI(\widetilde{K},\widetilde{F}) $ where the mapping $\widetilde{F}(\widetilde x) \equiv \widetilde q + \widetilde M \widetilde x$ is affine. $\widetilde K$ is a the projection of the polytope $ K $. 
To show that this VI is an AVI, it suffices to show that $ \widetilde{K} $ is also a polytope. For any set $ S $ we denote the set of its extreme points by $ \ext(S). $
\begin{lemma} \label{lem:polytope}
Let $K \subseteq \Real^n, M \in \Real^{n\times n}, q \in \Real^n$ be a polytope and  suppose $ \AVI(K,M,q) $ is provided as an input to Algorithm A. Then the following are true,
\begin{enumerate}
\item The set $ \widetilde{K} $ generated in Step \ref{step:proj} of Algorithm A is also a polytope.
\item $ \AVI(\widetilde{K},\widetilde{M},\widetilde{q}) $ admits a solution. 
\end{enumerate}
\end{lemma}
\begin{proof} To show part 1, it suffices to show that there is a finite set of points in $ \Ktilde $ such that every point in $\widetilde K$ can be expressed as a convex combination of these points. To this end let $\ext(K)$ denote the set of extreme points of $K$. Since $ K $ is a polytope, $ \eta:=|\ext(K)|<\infty. $ Then from Step \ref{step:proj} of Algorithm A, we have, $\forall \widetilde x \in \widetilde K$, $\exists x \in K$ such that $\widetilde x =  R\t x$. Since $K$ itself is a polytope, $\exists \alpha_1,\hdots,\alpha_\eta \ge 0,$ $\sum_i \alpha_i = 1$ such that $x = \sum_{i} \alpha_i x_i$, where  $x_1,\hdots,x_\eta$ are the extreme points of $ K $. This implies, $\widetilde x = \sum_i \alpha_i \sqrt{\frac{n}{k}}R\t  x_i$. Thus, any vector $\widetilde x \in \widetilde K$ can be expressed as a convex combination of the points in the set $S = \left\{\sqrt{\frac{n}{k}}R\t x\ |\ x \in \ext(K)\right\}$.  Since $ S $ is finite, $ \Ktilde $ is a polytope. Consequently, $ \Ktilde $ is also compact. Standard results~\cite{facchinei03finiteI} now show that $ \AVI(\Ktilde,\Mtilde,\qtilde) $ admits a solution.
\end{proof}

\subsection{Problem reformulation}
For a point to be a solution of  a VI, by definition, \textit{(uncountably) infinitely many} inequalities must simultaneously hold. However for an AVI on a polytope, a reduction to a finite set of inequalities is possible. This important reformulation is formalized in the following lemma.

\begin{lemma} \label{lem:reform}
Let $K \subseteq \Real^n$ be a polytope and  
$q \in \mathbbm{R}^n$, $M \in \mathbbm{R}^{n \times n}$. A vector $x^* \in K$ is a solution of $\AVI(K, M,q)$ if and only if:
\begin{equation}
(x_e-x^*)\t (q+Mx^*) \ge 0, \qquad \forall x_e  \in \ext(K), \label{eq:reform}
\end{equation}•
where $\ext(K)$ denotes the set of all the extreme points of $K$.
\end{lemma}
\begin{proof}
``$ \implies $'' 
Let $x^* \in K$ solve $\AVI(K, q, M)$. Thus,
\begin{equation*}
(y-x^*)\t (q+Mx^*) \ge 0, \quad \forall y \hspace{2pt} \in K.
\end{equation*}
Put $y =x_e$ where $ x_e \in \ext(K). $Thus,
\begin{equation*}
(x_e-x^*)\t (q+Mx^*) \ge 0, \quad \forall x_e \hspace{2pt} \in \ext(K).
\end{equation*}
Hence, $x^*$ satisfies \eqref{eq:reform}. 

``$ \impliedby $'' Let $ \ext(K) = x_1,\hdots,x_\eta $ where $ \eta = |\ext(K)|. $
Consider an arbitrary vector $y \in K$. Since $ K $ is a polytope, there exists $\{ \alpha_i \}^{\eta}_{i =1}$, such that $\sum_{i=1}^{\eta} \alpha_i = 1$, $\alpha_i \geq 0, \forall i \in 1,\hdots,\eta$, and $y = \sum_{i=1}^{\eta} {\alpha_i}x_i.$
Since $x^*$ solves \eqref{eq:reform}, we get  for all $i $, 
\begin{align*}
\alpha_i(x_i - x^*)\t (q+Mx^*) &\ge 0. 
\end{align*}
Summing all the inequalities over $ i=1,\hdots,\eta $ and using that $\sum_{i=1}^{\eta} \alpha_i = 1$, we get that
\begin{equation*}
(y-x^*)\t (q+Mx) \ge 0.
\end{equation*}
Since $y \in K$ was an arbitrary vector, the above inequality is true for every $y \in K$. Hence, $x^*$ also solves $\AVI(K, M,q)$. 
\end{proof}

\subsection{Error analysis}
The lower dimensional problem (AVI$(\widetilde K,\widetilde M,\widetilde q)$) as constructed in the algorithm is the following. 
$$ \viproblem{AVI$(\Ktilde,\Mtilde,\qtilde)$}{\ Find $ \xtilde \in \Ktilde $ such that}{(y-\xtilde)\t(\Mtilde \xtilde+\qtilde)\geq 0\ \qquad\forall y \in \Ktilde.} $$
We have the following claim about the deterministic preimage of the solution of the problem $ \AVI(\Ktilde,\Mtilde,\qtilde). $

\begin{lemma} \label{lem:bound1}
Let $ \epsilon \in (0,1) $ and suppose $ \AVI(K,M,q) $ is provided as input to Algorithm A. 
Let $\widetilde x$ be the solution of $\AVI(\widetilde K, \widetilde M, \widetilde q)$ generated by Step \ref{step:detpre} of Algorithm A and  let $x_o$ be its  deterministic preimage. Then with  probability greater than $1 -  4\eta e^{-(\frac{\epsilon^2}{2} - \frac{\epsilon^3}{3})k/2}$, $ x_o $ satisfies 
\begin{equation*}
(x_e - x_o)\t (q + Mx_o) \ge \frac{n}{k}(x_e - x_o)\t  RR\t(M-M')x_o  
-\epsilon \cdot \|x_e - x_o\| \|q + Mx_o\|
\end{equation*}
for every $x_e \in \ext(K)$,
 where $M' = MRR\t$ and $\eta = |\ext(K)|$.
\end{lemma}

\begin{proof}
Let $\widetilde x$ be a solution to  $\AVI(\widetilde K, \widetilde M, \widetilde q)$ and $x_o$ be its deterministic preimage. This implies,
\[(\widetilde y-\widetilde x)\t (\widetilde q + \widetilde M \widetilde x) \ge 0, \forall \widetilde y \in \widetilde K.  \]
Let $M' = MRR\t$, then it follows from definition of $ K $ that,
\[\sqrt{\frac{n}{k}}(R\t (y - x_o))\t  \sqrt{\frac{n}{k}}R\t (q+M'x_o) \ge 0, \quad\forall y \in K.  \]
Consequently, for each $ x_e \in \ext(K) $,
\[\sqrt{\frac{n}{k}}(R\t (x_e - x_o))\t \sqrt{\frac{n}{k}}R\t (q+M'x_o) \ge 0. \]
On adding ($\sqrt{\frac{n}{k}}(R\t (x_e - x_o))\t \sqrt{\frac{n}{k}}R\t(M-M')x_o $) on both sides we get that for any $x_e \in \ext(K)$,
\begin{align}
\sqrt{\frac{n}{k}}(R\t (x_e - x_o))\t  \sqrt{\frac{n}{k}}R\t (q+Mx_o) 
 &\ge \sqrt{\frac{n}{k}}(R\t (x_e - x_o))\t  \sqrt{\frac{n}{k}}R\t(M-M')x_o  \label{eq:temp1}\\
&= \frac{n}{k}(x_e - x_o)\t RR\t(M-M')x_o. \label{eq:temp2}
\end{align}
Observe that the left hand side of \eqref{eq:temp1} is the an inner product of the random projection of vectors $(x_e-x_o)  $ and $ q+Mx_o $. Hence by Lemma \ref{lem:innerpresmult},  with probability greater than $1 -  4\eta e^{-(\frac{\epsilon^2}{2} - \frac{\epsilon^3}{3})k/2}$, $ x_o $ satisfies,
\begin{multline*}
(x_e - x_o)\t (q + Mx_o) \\ 
\ge \sqrt{\frac{n}{k}}(R\t (x_e - x_o))\t \sqrt{\frac{n}{k}}R\t (q+Mx_o)  -\epsilon  \|x_e - x_o\| \|q + Mx_o\| ,\quad \forall x_e \in \ext(K). 
\end{multline*}
Now using \eqref{eq:temp2}, we get the result.
\end{proof}

For $\AVI(K,M,q)$ and $ R,x_o,\epsilon $ as in Lemma \ref{lem:bound1} and a point $ x_e \in \ext(K), $ define the quantities, 
\begin{align}
\mu_1(x_e) &:= \frac{n}{k}(x_e - x_o)\t RR\t(M-M')x_o, \label{eq:mu1}\\
\mu_2(x_e) &:= -\epsilon \cdot \|x_e - x_o\|\cdot \|q + Mx_o\|. \label{eq:mu2}
\end{align}
Below we derive bounds on $ \mu_1,\mu_2. $

To this end, recall that any matrix $ A \in \Real^{m\times n} $ admits a unique \textit{pseudoinverse}, which is a matrix $A^{+} \in \Real^{n\times n}$ that satisfies a set of conditions~\cite{laub2005matrix}. Furthermore, if $ A $ has full column rank, then $ A^+=(A\t A)\inv A\t. $ By Lemma \ref{lem:rank} and \ref{lem:subspace} the random matrix $ R $ constructed in Steps {\bf R 1 - R 2} 
has full column rank, whereby $(R\t R)^{-1}R\t =R\t$ is the pseudoinverse of $R$. Furthermore, by the singular value decomposition~\cite{horn90matrix}, since $R $ has rank $ k $, there exist orthogonal matrices $U \in \mathbbm{R}^{n \times n}$ and $V \in \mathbbm{R}^{k \times k}$ (\ie, $U\t U = I, V\t V = I$), such that $$R = U\Sigma V\t,$$
where $\Sigma$ = $\bigl(\begin{smallmatrix} S&0\\ 0&0 \end{smallmatrix} \bigr) \in \mathbbm{R}^{n \times k}$, $S$ = diag($\sigma_1,\dots, \sigma_k$) $\in \mathbbm{R}^{k \times k}$, and $\sigma_1 \geq\cdots \geq\sigma_{k} > 0$. This is because, out of the singular values of $ R $ (diagonal entries of $ \Sigma $), exactly $ k $ (the rank of $ R $) values $\sigma_1,\hdots,\sigma_k  $ must be nonzero~\cite{horn90matrix}. 
Furthermore, the pseudoinverse $R^+=R\t$ of $ R $ is given as~\cite{laub2005matrix},
\begin{equation*}
R^+ = V\Sigma^+ U\t ,
\end{equation*}
where $\Sigma^+$ = $\bigl(\begin{smallmatrix} S^{-1}&0\\ 0&0 \end{smallmatrix} \bigr) \in \mathbbm{R}^{k \times n}$. 

Finally, recall that the $ \ell_2 $-induced norm on a matrix is \textit{unitarily invariant}~\cite[p.\ 346, 357]{horn90matrix}, \ie,  
if $P \in \mathbbm{R}^{m \times m}$ and $Q \in \mathbbm{R}^{m \times m}$ are two orthogonal matrices, i.e., $P\t P = I$ and $Q\t Q = I$ then for any $A \in \mathbbm{R}^{m \times m}$,   
$ \|PAQ\| =  \|A\|. $

\begin{lemma} \label{lem:bound2}
Let  $\AVI(K,M,q)$ and $ R,x_o,\epsilon $ be as in Lemma \ref{lem:bound1}, let $ x_e \in \ext(K) $ and let
$ \mu_1,\mu_2 $ be defined as in \eqref{eq:mu1} and \eqref{eq:mu2}. Then the following bounds hold:
\begin{enumerate}
\item $\mu_1(x_e) \ge - \frac{n}{k}D\cdot  \|M\| \cdot B$
\item $\mu_2(x_e) \ge -\epsilon D \|q\| - \epsilon D \|M\|B$
\end{enumerate}
where $D := \max_{x_1, x_2 \in K}\|x_1-x_2\|$ and $ B:= \max_{x \in K} \norm{x} $.
\end{lemma}

\begin{proof}
By Cauchy-Schwartz inequality,
\begin{align*}
\mu_1(x_e) &\ge - \frac{n}{k}\|x_e - x_o\| \cdot \|RR\t(M-M')x_o \| \\
&\ge - \frac{n}{k}\cdot D \cdot \norm{M}\cdot \|RR\t\|\cdot \|(I-RR\t)\|\cdot B
\end{align*}
where $ D = \underset{x_1, x_2 \in K}{\max}\|x_1-x_2\| $ and $ B=\max_{x \in K}\norm{x}. $ Let the singular value decomposition of $ R $ be $ R= U\Sigma V\t $, where $ U,\Sigma,V $ are as above. Consequently, 
\begin{align*}
\|RR\t\|=\|U\Sigma V\t V\Sigma^+ U\t \| \buildrel{(a)}\over= \|U\t U\Sigma \Sigma^+ U\t U \|  = \norm{\Sigma\Sigma^+},
\end{align*}
where $ (a) $ follows from noting that $ V\t V=I $ and the unitary invariance of the $ \ell_2 $-induced norm. Similarly,
\begin{align*}
\|(I-RR\t)\|= \|I-  U\Sigma V\t V\Sigma^+ U\t \| 
=\|U\t U-  U\t U\Sigma \Sigma^+ U\t U\| 
=\|I-  \Sigma \Sigma^+ \|.
\end{align*}
 Since $R$ has rank $ k $ (cf. Lemma \ref{lem:rank}), $\Sigma \Sigma^+$ is an $n\times n$ matrix with only $k$ of its diagonal entries as 1 and the remaining all entries being 0.  Therefore $ \norm{\Sigma \Sigma^+}=1 $. Likewise, $ \norm{I-\Sigma\Sigma^+}=1. $ This gives the required bound on $ \mu_1. $
 
For the other term $\mu_2$  we have by triangle inequality and by definitions of $ D,B, $
\begin{align*}
\mu_2(x_e) &\ge -\epsilon \cdot D \cdot (\|q\| + \|Mx_o\|) \ge -\epsilon \cdot D \cdot \|q\| - \epsilon \cdot D \cdot  \|M\|\cdot B,
\end{align*}
which completes the proof.
\end{proof}

Recall that $ D,B $ in the above lemma are finite since $ K $ is compact (cf. Lemma~\ref{lem:polytope}). In the following section, we complete the proof of Theorem \ref{thm:main}.

\subsection{Proof of Theorem \ref{thm:main}}
To prove Theorem \ref{thm:main}, we require the following result due to Candes and Tao~\cite{candes2006near} on optimal recovery from random measurements.

\begin{theorem} \label{thm:taorecov}
Suppose that $f \in \mathbbm{R}^n$ obeys $\|f\|_1 \le C_1$, and let $\alpha >0$ be a sufficiently small number (less than an absolute constant). Assume that we are given $k$ random measurements $y_i = \langle f, \Psi_i \rangle$, where $i \in \Omega$, $|\Omega| = k$ and $\{\Psi_i\}$ is a set of $k$- uniformly random orthonormal vectors. Then with probability 1, we have a unique minimizer $f^{\#}$ to the following problem:
\begin{equation*}
\underset{x \in \mathbbm{R}^n }{\text{min }} \|x\|_{1} \hspace{3pt} \text{ \st \quad $y_i = \langle x, \Psi_i \rangle, \quad \forall i=1,\hdots,k.$}
\end{equation*}
Furthermore, with probability at least $1 - O(n^{-1/\alpha})$, we have the approximation
\begin{equation*}
\|f - f^{\#}\|_2 \le C \cdot C_1 \cdot (k/\ln(n))^{-1/2}.
\end{equation*}
Here, $C$ is a fixed constant depending on $\alpha$ but not on anything else. The implicit constant in $O(n^{-1/\alpha})$ is allowed to depend on $\alpha$.
\end{theorem}

\noindent We now complete the proof of our main result, Theorem \ref{thm:main}.\\

\begin{proof}[Proof of Theorem \ref{thm:main}]
Let $\widetilde x$ be generated by Step \ref{step:detpre} by solving the lower-dimensional $ \AVI(\Ktilde,\Mtilde,\qtilde) $ and let $x_o \in K$ be a deterministic preimage of $ \xtilde $. From Lemma~\ref{lem:bound1} and Lemma~\ref{lem:bound2} that with probability greater than $p := (1 -  4\eta e^{-(\frac{\epsilon^2}{2} - \frac{\epsilon^3}{3})k/2})$, $ x_o $ satisfies
\begin{equation}
(x_e - x_o)\t (q + Mx_o) \ge  - \frac{n}{k} \|M\| B  - \epsilon \cdot  D \|q\| - \epsilon \cdot  D  \|M\| B \label{eq:temp0}, \quad \forall x_e \in \ext(K).
\end{equation}
Consider any $ y\in K $. There exist $ \alpha_1,\hdots,\alpha_\eta  \geq 0$ such that $ \sum_i \alpha=1 $ such that $ y=\sum_i x_i $ where $ x_1,\hdots,x_\eta $ are the extreme points of $ K. $ Multiplying the inequality in \eqref{eq:temp0} corresponding to each $ x_i \in \ext(K) $ by $ \alpha_i $, and adding over all $ i $, we get that under the event that \eqref{eq:temp0} is true, the event
\begin{equation}
(y - x_o)\t (q + Mx_o) \ge  - \frac{n}{k} \|M\| B  - \epsilon \cdot  D \|q\| - \epsilon \cdot  D  \|M\| B, \label{eq:temp3}\quad \forall y \in K, 
\end{equation}
is true. 
Consequently, $ x_o $ satisfies \eqref{eq:temp3} with probability at least $ p $. 
Since $ \delta $ is the confidence parameter, we set $p > (1-\delta)$. This necessitates that $ k > \frac{2\ln(\frac{4\eta}{\delta})}{(\frac{\epsilon^2}{2} - \frac{\epsilon^3}{3})} $. Consequently, with probability strictly greater than $ (1-\delta), $ $ x_o $ satisfies,
\begin{equation*}
(y - x_o)\t (q + Mx_o) \ge  - \frac{\left(\frac{\epsilon^2}{2}-\frac{\epsilon^3}{3}\right)n}{2\ln(\frac{4\eta}{\delta})} \|M\| B  - \epsilon \cdot  D \|q\|  
- \epsilon \cdot  D  \|M\| B, \quad \forall y \in K.
\end{equation*}
This proves the first statement of the theorem.

To prove the second statement, let $ x^* $ be the vector generated by Step~\ref{step:l1} of Algorithm A and let $ \alpha>0 $ be a small number as required by Theorem \ref{thm:taorecov}. By Theorem \ref{thm:taorecov}, we have with probability at least $1 - O(n^{-1/\alpha})$,
\begin{equation}
\|x_0 - x^*\| \le C  \|x_0\|_1  \leq CB' \left(\frac{k}{\ln(n)}\right)^{-1/2}, \label{eq:xo1}
\end{equation}
where $ B':= \max_{x\in K} \norm{x}_1. $ Let $ x^\# $ be the output of Algorithm A. Since $ x_o \in K$ and since the $ \ell_2 $-projection on the closed convex set $ K $ in Step \ref{step:projk} is non-expansive~\cite{facchinei03finiteI}, we have
\begin{align}
\|x_o - x^{\#}\| &= \norm{\Pi_K(x_o) - \Pi_K(x^*)} \le \|x_o - x^*\|. \label{eq:xo2}
\end{align}
Combining \eqref{eq:xo1} and \eqref{eq:xo2}, we get the second statement. The proof is complete.
\end{proof}

With this we conclude the theoretical portion of the paper. In the following section we point out some remarks about the algorithm, following which we present numerical results.

\section{Some remarks about the algorithm} \label{sec:remarks}

\subsection{Solving the lower dimensional AVI} \label{sec:lowerdim}
Although the lower-dimensional problem $ \AVI(\Ktilde,\Mtilde,\qtilde) $ is indeed an AVI, there is a practical difficulty in processing it. AVI solvers typically require the polyhedron defining the constraints to be given in its half-space representation and we do not have direct access to the  half-space representation of $\widetilde K$.  Following are two approaches to this issue. \\

\noindent\textbf{First approach:} We augment extra variables and convert the lower-dimensional AVI to a sparse larger-dimensional AVI. 
 Define a new augmented variable $ v = (\xtilde,x) \in \Real^{n+k}$.  
Instead of $ \AVI(\Ktilde,\Mtilde,\qtilde)$, the solver is supplied problem $\AVI(\hat K, \hat M, \hat q)$, where 
$$\hat q = \bigl[\begin{smallmatrix} \widetilde{q}\\ 0^{n \times 1} \end{smallmatrix} \bigr], \quad \hat M = \bigl[\begin{smallmatrix} \widetilde{M} & 0^{k \times n}\\ 0^{n \times k} & 0^{n \times n} \end{smallmatrix} \bigr], \quad \aur \quad \hat K = \{v \in \mathbb{R}^{n+k} | Cv = 0, Dv \le b\},$$ where $C= [I | -R\t ], D = [0 | A]\}$.
Although the problem now becomes a larger dimensional one, there is an enormous amount of sparsity in the formulation which can potentially be exploited by AVI solvers to solve this problem rather quickly. We have found this to be the case with the PATH solver and have used it in our numerical results.\\

\begin{table*}[!t]
\begin{center}
\resizebox{1\textwidth}{!} {
\begin{tabular}{c c c c c c c c c c}

$ n $	&	$ k $	&	$ m $	&	Natural Map Residual	&	Angle	&	Difference Norm	&	Major Low	&	Minor Low	&	Major High	&	Minor High \\ \hline 
100	&	5	&	10	&	8.68	&	142.20	&	5.34	&	4	&	9	&	7	&	2468	\\		
100	&	10	&	10	&	9.22	&	141.75	&	2.50	&	4	&	9	&	7	&	2468	\\		
100	&	30	&	10	&	7.84	&	141.38	&	1.80	&	4	&	11	&	7	&	2468	\\		
100	&	50	&	10	&	7.68	&	141.05	&	5.90	&	5	&	15	&	7	&	2468	\\		
100	&	70	&	10	&	6.43	&	142.14	&	7.27	&	5	&	15	&	7	&	2468	\\		
100	&	90	&	10	&	4.87	&	141.71	&	11.71	&	7	&	239	&	7	&	2468	\\	\hline 
150	&	5	&	15	&	12.12	&	141.55	&	6.09	&	4	&	12	&	3	&	1736	\\		
150	&	10	&	15	&	11.26	&	140.69	&	1.30	&	4	&	12	&	3	&	1736	\\		
150	&	30	&	15	&	11.05	&	141.94	&	2.96	&	5	&	16	&	3	&	1736	\\		
150	&	50	&	15	&	10.51	&	141.65	&	1.58	&	5	&	15	&	3	&	1736	\\		
150	&	80	&	15	&	9.03	&	141.70	&	2.83	&	6	&	21	&	3	&	1736	\\		
150	&	110	&	15	&	7.42	&	142.21	&	1.42	&	5	&	160	&	3	&	1736	\\		
150	&	135	&	15	&	5.27	&	140.33	&	6.36	&	7	&	113	&	3	&	1736	\\		\hline 
200	&	5	&	16	&	12.68	&	141.68	&	1.46	&	5	&	15	&	2	&	2201	\\
200	&	10	&	16	&	12.39	&	141.59	&	1.47	&	4	&	13	&	2	&	2201	\\
200	&	30	&	16	&	12.33	&	142.22	&	2.13	&	5	&	19	&	2	&	2201	\\
200	&	50	&	16	&	12.32	&	141.61	&	1.90	&	5	&	22	&	2	&	2201	\\
200	&	80	&	16	&	11.65	&	141.37	&	2.18	&	6	&	29	&	2	&	2201	\\
200	&	110	&	16	&	9.92	&	141.44	&	1.29	&	5	&	21	&	2	&	2201	\\
200	&	140	&	16	&	8.74	&	140.94	&	3.74	&	6	&	53	&	2	&	2201	\\
200	&	180	&	16	&	5.51	&	141.28	&	3.13	&	8	&	128	&	2	&	2201	\\	\hline	
250	&	5	&	25	&	15.03	&	141.87	&	2.27	&	4	&	14	&	3	&	3031	\\		
250	&	10	&	25	&	14.75	&	141.22	&	4.22	&	4	&	17	&	3	&	3031	\\		
250	&	30	&	25	&	14.98	&	141.75	&	2.43	&	4	&	19	&	3	&	3031	\\		
250	&	50	&	25	&	14.16	&	141.28	&	3.82	&	5	&	22	&	3	&	3031	\\		
250	&	70	&	25	&	13.84	&	141.65	&	2.33	&	5	&	28	&	3	&	3031	\\		
250	&	100	&	25	&	13.20	&	141.03	&	6.31	&	6	&	45	&	3	&	3031	\\		
250	&	130	&	25	&	12.13	&	140.69	&	2.42	&	5	&	28	&	3	&	3031	\\		
250	&	160	&	25	&	10.88	&	141.25	&	2.33	&	5	&	27	&	3	&	3031	\\		
250	&	190	&	25	&	9.49	&	141.31	&	13.29	&	6	&	230	&	3	&	3031	\\		
250	&	225	&	25	&	2.58	&	141.49	&	496.15	&	8	&	140	&	3	&	3031	\\	\hline	

\end{tabular}}
\end{center}
\caption{Performance of Algorithm A on randomly generated test problems where each entry of $ M,q,A,b $ was chosen from $ N(0,1) $.} \label{tab:approx}
\end{table*}

\begin{table*}
\begin{center}
\resizebox{1\textwidth}{!} {

\begin{tabular}{c c c c c c c c c c}

$ n $	&	$ k $	&	$ m $	&	Natural Map Residual	&	Angle	&	Difference Norm	&	Major Low	&	Minor Low	&	Major High	&	Minor High	\\	\hline
100	&	5	&	10	&	2.54	&	131.67	&	1.08	&	4	&	9	&	2	&	1132	\\		
100	&	10	&	10	&	2.55	&	123.58	&	1.04	&	4	&	9	&	2	&	1132	\\		
100	&	30	&	10	&	2.44	&	122.26	&	1.97	&	5	&	11	&	2	&	1132	\\		
100	&	50	&	10	&	2.01	&	125.92	&	1.01	&	5	&	12	&	2	&	1132	\\		
100	&	70	&	10	&	1.82	&	119.70	&	2.76	&	8	&	17	&	2	&	1132	\\		
100	&	90	&	10	&	1.15	&	107.13	&	125.55	&	7	&	911	&	2	&	1132	\\		\hline
150	&	5	&	15	&	3.06	&	124.51	&	1.33	&	4	&	11	&	7	&	3404	\\		
150	&	10	&	15	&	3.17	&	123.56	&	2.02	&	4	&	12	&	7	&	3404	\\		
150	&	30	&	15	&	3.13	&	118.31	&	2.82	&	5	&	14	&	7	&	3404	\\		
150	&	50	&	15	&	3.04	&	120.46	&	3.37	&	6	&	18	&	7	&	3404	\\		
150	&	80	&	15	&	2.58	&	120.34	&	2.70	&	6	&	21	&	7	&	3404	\\		
150	&	110	&	15	&	2.02	&	108.59	&	4.24	&	7	&	16	&	7	&	3404	\\		
150	&	135	&	15	&	1.36	&	109.33	&	4.93	&	5	&	47	&	7	&	3404	\\		\hline	
200	&	5	&	16	&	3.59	&	119.97	&	1.03	&	4	&	9	&	3	&	2332	\\		
200	&	10	&	16	&	3.59	&	123.81	&	1.09	&	5	&	13	&	3	&	2332	\\		
200	&	30	&	16	&	3.75	&	115.21	&	12.22	&	6	&	214	&	3	&	2332	\\		
200	&	50	&	16	&	3.63	&	120.52	&	1.24	&	5	&	15	&	3	&	2332	\\		
200	&	80	&	16	&	3.35	&	118.66	&	1.41	&	6	&	16	&	3	&	2332	\\		
200	&	110	&	16	&	3.05	&	114.44	&	1.27	&	6	&	18	&	3	&	2332	\\		
200	&	140	&	16	&	2.35	&	113.24	&	165.33	&	8	&	33	&	3	&	2332	\\		
200	&	180	&	16	&	1.62	&	106.38	&	1.84	&	8	&	1287	&	3	&	2332	\\	\hline	
250	&	5	&	25	&	4.14	&	119.08	&	1.44	&	4	&	14	&	3	&	2929	\\		
250	&	10	&	25	&	4.03	&	119.81	&	1.32	&	4	&	16	&	3	&	2929	\\		
250	&	30	&	25	&	4.15	&	117.70	&	2.12	&	5	&	13	&	3	&	2929	\\		
250	&	50	&	25	&	4.08	&	116.17	&	1.70	&	6	&	20	&	3	&	2929	\\		
250	&	70	&	25	&	3.90	&	118.32	&	1.50	&	6	&	29	&	3	&	2929	\\		
250	&	100	&	25	&	3.61	&	115.63	&	9.64	&	7	&	34	&	3	&	2929	\\		
250	&	130	&	25	&	3.39	&	114.22	&	2.15	&	8	&	37	&	3	&	2929	\\		
250	&	160	&	25	&	2.97	&	111.25	&	4.20	&	8	&	364	&	3	&	2929	\\		
250	&	190	&	25	&	2.57	&	109.61	&	2.92	&	8	&	79	&	3	&	2929	\\		
250	&	225	&	25	&	1.79	&	104.41	&	1.49	&	7	&	1381	&	3	&	2929	\\	\hline	
\end{tabular}}
\end{center}
\caption{Performance of Algorithm A on randomly generated test problems where each entry of $ M,q,A,b $ was chosen from $ U[0,1] $.} \label{tab:approxunif}
\end{table*}


\noindent\textbf{Second approach:} Alternatively, one may try and deduce the half-space representation of $\widetilde K$. 
To do so, first, complete the random projection matrix $R$ by filling in the remaining entries (denote these entries by $\Delta^{n \times (n-k)}$) by picking the remaining set of $(n-k)$ uniformly random orthonormal vectors in exactly the same manner as described in Section \ref{sec:constR}. Let this complete $ n\times n $ matrix be denoted by $\hat{R}$. If $ K=\{x \in \Real^n| Ax \leq b\}, $ let $ \hat K=\{y \in \Real^n| \hat A y \leq b \}, $ where $ \hat A = \sqrt{\frac{k}{n}}A((\hat{R}\t )^{-1}. $  Note that $\hat{R}=[R   \Delta^{n \times (n-k)}]$, which implies that the first $ k $ components of a vector $y \in \hat K$ comprise the vector $\widetilde{x} = \proj x$ which is an element of $ \Ktilde $. 

Thus the required polytope $ \Ktilde $ is the Euclidean projection of $ \hat K $ on the first $ k $ components. To obtain the half-space representation for this polytope one may make use of a method such as Fourier-Motzkin elimination~\cite{dantzig1973fourier}. Unfortunately, Fourier-Motzkin elimination is known to have poor complexity and could potentially nullify any advantages of dimensionality reduction. 

Converting the half-space representation of a polyhedron to its vertex-representation and vice-versa is a fundamental combinatorial problem. 
Indeed, a side-story of our algorithm is the  bringing to fore of the combinatorial nature of polyhedra and indeed of the AVI, which has otherwise been suppressed in ``continuous'' optimization efforts.

\subsection{On the lower-dimension}
The lower dimension $ k $ is required to be of the order of $ \ln (\eta/\delta)/ \epsilon^2 $. By introducing slack variables if necessary, we may assume without loss of generality that the polytope $ K $ is represented as $ \{x' \ |\ A' x'=b', x'_I\geq 0 \} $ where $ x'_I=(x'_i)_{i\in I}$ and that the given $\AVI(K,M,q)  $ is specified in the space of $ x' $. Suppose $ n $ the dimension of $ x'. $ If $ A' $ is full row rank, then $ \eta = |\ext(K)| \leq \binom{n}{m}$ where $ m $ is the number of rows of $ A' $ (this follows from the argument used for bounding the number of basis feasible solutions in a linear program; see example~\cite[Ch. 3]{bazaraa2011linear}). Consequently, for fixed $ m,$ we have $ \ln(\eta) =O(\ln n)$ and the lower dimension $ k = O(\ln n/\delta)/\epsilon^2 $. In general, by Stirling's approximation, $ \ln \binom{n}{m} \simeq n H\left(\frac{m}{n}\right)$ where $ H(t) \equiv -t \ln t -(1-t)\ln (1-t) $ is the entropy function~\cite[p.\ 2]{mackay2003information}, and hence $ \ln \binom{n}{m} $ is at most $ n $ (this is an approximate statement, the quality of which depends on the quality of Stirling's approximation).

The term $ \eta $ appears because to solve the AVI, we have to satisfy $ \eta $-many inequalities, and $ \eta $-many inner products have to be \textit{simultaneously} preserved in the JL lemma. Notice, however, that each fixed $ y \in K$ can be written as a convex combination of $ n+1 $ points from $ \ext(K) $ (by Caratheodory's theorem~\cite{rockafellar97convex}). Thus for each $ y \in K $, we only need to simultaneously preserve $ n+1 $ inner products. Thus for any $ y\in K $ the inequality
\[ (y-x_o)\t (Mx_o+q) \geq 0, \]
holds with probability $ 1-4(n+1)e^{-(\frac{\epsilon^2}{2}-\frac{\epsilon^3}{3})k/2} $, whereby for this event to hold with probability $ >1-\delta, $ we need $ k $ to be only of the order of $ \ln \left(\frac{4(n+1)}{\delta}\right)/\epsilon^2. $ It is for showing that ``$ (y-x_o)\t (Mx_o+q) \geq 0 $'' holds \textit{for all} $ y\in K $ that one requires $k \sim  \ln (\eta/\delta)/ \epsilon^2. $
One may interpret this issue also to be a manifestation of the combinatorial nature of the polyhedron.

Another way one may interpret this matter is via the V-representation of polyhedra~\cite{ziegler1995lectures}. In the V-representation the shape of $ K $ is defined by its $ \eta $ extreme points and $ \eta $ is thus the indicator of the ``complexity'' of this shape. On the other hand in the H-representation, the shape of the polyhedron is determined by the number of half-spaces necessary to describe it. 

\subsection{Construction of $\widetilde M$} 
Notice that while $\widetilde q$ and $\widetilde K$ are projections (in the sense of Section~\ref{sec:construction}) of $ q $ and $ K $, respectively,  $\widetilde M$ is \textit{not} a projection of $ M. $ Rather $ \Mtilde $ may be viewed as the projection of a matrix that is akin to a least squares approximation to $ M. $ We explain this below. 

For the lower-dimensional problem to be an AVI,  we require that in addition to $ \Mtilde $ being a projection of some matrix, $ \widetilde{M} $ must also be compatible with right-multiplication by $ \xtilde=\proj x. $ Thus $ \Mtilde $ must be of the form $ \Mtilde =  R\t M' $, where $ M' =XR\t$ and $ X $ is the matrix to be determined. Now since the eventual error depends on $\|M - M'\|$, in our construction we further let $M'$ to be of the form $MZR\t $ where now $Z$ is to be determined. To minimize the error, we require $Z$ to  be such that  $ZR\t $ is as close to the identity as possible, \ie, $ Z $ must solve the following ``least-squares'' optimization problem:
\begin{align*}
\underset{Z \in \mathbbm{R}^{n \times k}}{\min} \|ZR\t  - I\|^{2}.
\end{align*}
The solution of the above problem is  $Z  = R(R\t R)^{-1} = R$, since $ R $ is orthonormal. This gives the construction of $ \Mtilde $ as $ \Mtilde=R\t M R. $ In addition to being suitable for minimizing the error, we note that $ \Mtilde $ enjoys the property that $ \Mtilde $ is positive semidefinite if $ M $ is positive definite.

\section{Numerical Results} \label{sec:sim}
We now present numerical results to show how our algorithm performs in practice.  

To keep the trials generic, test problems were generated randomly. The input to the algorithm was given as $ \AVI(K,M,q) $ where we took $K = \{ x \in \mathbbm{R}^n | Ax \le b, L \le x \le U\}$. The bounds $L,U$ on $x$ were introduced to ensure that $K$ is compact.  
To generate a random input problem, all the entries of $q$, $M$, $A$ and $b$ were generated randomly with entries drawn from distributions $ N(0,1) $ and $ U[0,1] $; $ -L,U $ were set to be large.
We applied our algorithm with various choices of the lower dimension $ k $, for each value of the higher dimension $ n. $ All AVIs were solved using the PATH solver~\cite{pathweb}. 

The sections below show the performance of the algorithm for two applications. Section \ref{sec:approx} discusses results for obtaining an approximate solution to the given AVI (\ie, Algorithm A). In Section \ref{sec:exact}, we apply Algorithm A to solve the given AVI exactly by using Algorithm A to generate an  initial point for the solver.

\subsection{Performance for an approximate solution} \label{sec:approx}
Table~\ref{tab:approx} contains results pertaining to Algorithm A. To benchmark the performance of the algorithm, the given AVI was approximately by Algorithm A and also solved exactly. Since the projections were random, the results reported for each test problem are the average of the results over 10 random choices of $ R $. Table~\ref{tab:samples} shows the behavior across different random trials for a representative test problem.

\begin{table*}[!t]
\begin{center}
\resizebox{1\textwidth}{!} {

\begin{tabular}{c c c c c c c c c c}

$ n $	&	$ k $	&	$ m $	&	Natural Map Residual	&	Angle	&	Difference Norm	&	Major Low	&	Minor Low	&	Major High	&	Minor High	\\	\hline	
150	&	30	&	15	&	11.53	&	142.05	&	13.61	&	8	&	39	&	3	&	1736		\\
150	&	30	&	15	&	9.98	&	142.24	&	1.11	&	3	&	8	&	3	&	1736		\\
150	&	30	&	15	&	9.22	&	140.96	&	1.03	&	4	&	8	&	3	&	1736		\\
150	&	30	&	15	&	11.67	&	143.69	&	2.12	&	6	&	19	&	3	&	1736		\\
150	&	30	&	15	&	9.84	&	140.42	&	1.04	&	3	&	9	&	3	&	1736		\\
150	&	30	&	15	&	10.30	&	141.12	&	5.62	&	7	&	24	&	3	&	1736		\\
150	&	30	&	15	&	9.91	&	143.82	&	1.13	&	4	&	14	&	3	&	1736		\\
150	&	30	&	15	&	10.14	&	143.04	&	1.13	&	4	&	16	&	3	&	1736		\\
150	&	30	&	15	&	10.53	&	141.18	&	1.23	&	3	&	9	&	3	&	1736		\\
150	&	30	&	15	&	11.25	&	140.83	&	1.61	&	4	&	11	&	3	&	1736		\\
\hline 
\end{tabular}}
\end{center}
\caption{Results for Algorithm A with different samples of the random projection (for a test problem generated by taking each entry of $ M,q,A,b $ from $ N(0,1) $).} \label{tab:samples} 
\end{table*}

\begin{table*}[!t]
\begin{center}
\resizebox{1\textwidth}{!} {

\begin{tabular}{c c c c c c c c c}
\\ 
																										$ n $	&	$ k $	&	$ m $	&	Major High	&	Minor High	&	Major Additional	&	Minor Additional	&	Major Total	&	Minor Total	\\		\hline
100	&	5	&	10	&	40	&	9885	&	16	&	3699	&	19	&	3708	\\		
100	&	10	&	10	&	40	&	9885	&	2	&	690	&	6	&	700	\\		
100	&	25	&	10	&	40	&	9885	&	9	&	2478	&	14	&	2494	\\		
100	&	50	&	10	&	40	&	9885	&	10	&	2784	&	15	&	2802	\\		
100	&	75	&	10	&	40	&	9885	&	16	&	4088	&	21	&	4107	\\		\hline 
125	&	5	&	12	&	12	&	4354	&	4	&	1821	&	9	&	1834	\\		
125	&	10	&	12	&	12	&	4354	&	7	&	2873	&	12	&	2886	\\		
125	&	25	&	12	&	12	&	4354	&	11	&	4223	&	16	&	4236	\\		
125	&	50	&	12	&	12	&	4354	&	7	&	2693	&	12	&	2712	\\		
125	&	75	&	12	&	12	&	4354	&	8	&	3392	&	14	&	3414	\\		
125	&	100	&	12	&	12	&	4354	&	8	&	3185	&	14	&	3205	\\	\hline 	
150	&	5	&	15	&	19	&	7356	&	5	&	2574	&	10	&	2587	\\		
150	&	10	&	15	&	19	&	7356	&	8	&	3940	&	12	&	3952	\\		
150	&	30	&	15	&	19	&	7356	&	5	&	2521	&	9	&	2535	\\		
150	&	60	&	15	&	19	&	7356	&	4	&	2252	&	9	&	2272	\\		
150	&	75	&	15	&	19	&	7356	&	7	&	3468	&	13	&	3495	\\		
150	&	90	&	15	&	19	&	7356	&	4	&	2004	&	12	&	2037	\\		
150	&	120	&	15	&	19	&	7356	&	4	&	2011	&	10	&	2101	\\	\hline	
	
\end{tabular}}
\end{center}
\caption{Performance of the exact algorithm (test problems generated by choosing each entry of $ M,q,A,b $ independently from $ N(0,1) $)} \label{tab:exact}
\end{table*}

The columns of Tables \ref{tab:approx}, \ref{tab:approxunif} and \ref{tab:samples} contain the following entries.
\begin{itemize}
\item $n$ -- dimension of the $ \AVI(K,M,q) $ given as input to Algorithm A; $m$ -- number of rows of the matrix $ A $ in the definition of $K$; $k$ -- dimension to which we project the given AVI.
\item Natural Map Residual  $= \frac{\|\fnat_K(x^{\#}) \|}{\|x^{\#}\|+1} $, where $ x^\# $ is the output of  Algorithm A. This quantity should be $ 0 $ for $ x^\# $ to be a solution. We normalize by $ (\norm{x^\#}+1 )$ in order to allow for comparisons across difference values of $ n,k $.
\item Angle -- the largest angle between $ q+Mx^\# $ and a vector $ y-x^\# $ as $ y $ ranges over $ K $, where $ x^\# $ is as above. Define $\beta = \underset {y \in K}{\min}(y-x^{\#})\t (q + Mx^{\#})$ and let $y^*$ be the corresponding minimizer, then
\[\text{Angle} = \arccos\left(\frac{\beta}{\|y^* - x^{\#}\| \cdot \|q + Mx^{\#} \|}\right).\]
\item Difference Norm -- norm of the difference between $x^{\#}$ and the exact solution computed by solving the AVI directly (denoted $\bar x$), normalized by the norm of the exact solution, i.e., 
\[\text{Difference Norm} = \frac{\|x^{\#} - \bar x \|}{\|\bar x\|+1}. \] 
\item Major/Minor Low - number of major/minor iterations reported by PATH~\cite{pathweb} to obtain a solution using our algorithm
\item Major/Minor High - number of major/minor iterations reported by PATH~\cite{pathweb} to solve the higher dimensional problem directly
\end{itemize}

For every tuple $(n,m,k)$, 10 independent simulations were carried out, realizing a different random matrix $R$ each time. All the parameters in Table~\ref{tab:approx} and Table~\ref{tab:approxunif} are  the average values over the set of 10 trials. The results are as given in Table \ref{tab:approx} (for test problems generated from $ N(0,1) $) and Table \ref{tab:approxunif} (for test problems generated from $ U[0,1] $).

\subsubsection{Key observations}
The first observation to be made from Tables~\ref{tab:approx} and~\ref{tab:approxunif} is that the algorithm produces only an approximation, which is evident from the fact that the natural map residual is non-zero and ``Angle'' is greater than $90^o$.  However, for each $n$, the normalized natural map residual decreases monotonically with the value of lower dimension $k$. Clearly, as we increase the value of the lower dimension, the natural map residual decreases towards zero. Furthermore, in Table~\ref{tab:approxunif}, one sees a decrease in the ``Angle'' with increasing $ k $, for each $ n. $ Interestingly, this decrease is not seen in Table~\ref{tab:approx}. We do not know of a way of explaining this.

The most important observation is that for every case, the number of minor iterations (minor low) consumed by our algorithm is significantly lower than the case where the high dimensional problem is attempted to solve directly (minor high). Though results corresponding to major iterations are inconclusive on a whole, this validates our approach of finding a quick approximation. Notice that ``Difference norm'' is \textit{not} monotonic. This may be due to possible non-uniqueness of the solution of the AVI.

Recall that Tables~\ref{tab:approx} and \ref{tab:approxunif} are average values of multiple trials of $ R. $ For a fixed problem and a fixed $n$, $k$ and $m$, a representative set of results from different choices of $ R $ are reported in Table~\ref{tab:samples}. One can see that the variation across different trials is relatively small and the average values reported Tables~\ref{tab:approx} and \ref{tab:approxunif} are representative.

\subsection{Performance for an exact solution} \label{sec:exact}
For these results, the computation was carried out in two steps. First, an approximate solution was obtained using our algorithm, like in the previous section. Next, this vector was supplied as an initial point into the AVI solver and an exact solution to the original AVI was computed. The original problem was independently solved using the solver directly with a random initial point, and the performance was compared.  The results for this case have been tabulated in terms of the following parameters:
\begin{itemize}
\item $n, k, m$ and Major/Minor High as in Section \ref{sec:approx}
\item Major/Minor Total - total number of major/minor iterations reported by PATH to obtain a final solution with our approximate solution as an initial point
\item Major/Minor Additional - number of extra major/minor iterations reported by PATH to solve the higher dimensional problem {\it after} the initial point has been supplied from our algorithm
\end{itemize}

\subsubsection{Key observations}
Notice that the solution obtained in this case exactly solves our original high dimensional AVI. 
Minor Total is lower in all the test cases than the corresponding number for when the problem is solved directly. 
The additional iterations, both major and minor,  consumed on the higher dimensional problem, are significantly lower than the corresponding values for when the problem is solved directly. This implies that when our recovered solution is supplied as an initial point to the solver, it computes an exact solution faster than the direct case, where an initial point is generated randomly.

\section{Conclusions} \label{sec:disc}
Motivated by emerging problems in `Big Data', this paper has presented a new method for dimensionality reduction  of AVIs with compact feasible regions. The method yields with high probability an approximate solution to the given AVI by solving an AVI of a  lower dimension; the latter is formed by appropriately projecting the given AVI on a lower-dimensional space. Using the approximate solution as an intial point, the method can also be used to `hot start' a solver for the given problem and thereby find an exact solution. We presented numerical results to demonstrate that the method is indeed effective in practice.

\section*{Acknowledgments}
The authors would like to thank Dr. Dinesh Garg of IBM Research, India for his inputs on this topic.
\bibliographystyle{spmpsci}

\bibliography{references_bib}

\end{document}